\documentclass[12pt]{article}
\usepackage{amsmath, amsthm, amsfonts}
\usepackage{amssymb}
\usepackage{amscd}
\usepackage{verbatim}
\newcommand{\rlog}[1]{{\log\frac{1}{ #1}}}
\newcommand{\logfrac}{{\log\frac{1}{r}}}
\newcommand{\ud}{\mathrm{d}}

\newcommand{\eqdef}{\stackrel{{\mathrm def}}{=}}

\newcommand{\dx}{\,\mathrm{d}x}

\newcommand{\sob}{{H^1_{0;r}(B)}}

\newcommand{\be}{\begin{equation}}
\newcommand{\ee}{\end{equation}}

\newcommand{\bea}{\begin{eqnarray}}
\newcommand{\eea}{\end{eqnarray}}
\newcommand{\bean}{\begin{eqnarray*}}
\newcommand{\eean}{\end{eqnarray*}}

\newcommand{\wlim}{\mbox{ w-lim }}

\newcommand{\R}{{\mathbb R}}
\newcommand{\N}{{\mathbb N}}

\newtheorem{theorem}{Theorem}[section]

\newtheorem{lemma}[theorem]{Lemma}
\newtheorem{definition}[theorem]{Definition}
\newtheorem{remark}[theorem]{Remark}
\newtheorem{proposition}[theorem]{Proposition}

\catcode`@=11 \@addtoreset{equation}{section} \catcode`@=12

\begin{document}
 \title{Concentration profiles for the Trudinger-Moser functional
are shaped like toy pyramids\thanks{Research supported by Swedish Research Council and Wenner-Gren Foundations.}}
\author{David Costa\\
 {\small Department of Mathematical Sciences}\\
 {\small University of Nevada Las Vegas}\\
 {\small Las Vegas, NV 89154-4020}\\
{\small  USA}\\
{\small costa@unlv.nevada.edu}\\\and
\and Cyril Tintarev\thanks{Research done in part at Centre for Applicable Mathematics of TIFR}
\\{\small Department of Mathematics}\\{\small Uppsala University}\\
{\small P.O.Box 480}\\
{\small SE-751 06 Uppsala, Sweden}\\{\small
tintarev@math.uu.se}}
\date{}
\maketitle
\begin{abstract} This paper answers the conjecture of Adimurthi and
Struwe \cite{AdiStr}, that the semilinear Trudinger-Moser
functional
\begin{equation}
\label{J}
J(u)=\frac12\int_\Omega |\nabla u|^2\ud x-\frac{1}{8\pi}\int_\Omega (e^{4\pi u^2}-1)\ud x,
\end{equation}
(as well as functionals with more general critical
nonlinearities) satisfies the Palais-Smale
condition at all levels
except $\frac{n}{2}$, $n\in N$.
In this paper we construct critical sequences at any
level $c>\frac12$ corresponding to a large family of distinct concentration
profiles, indexed by all closed subsets $C$ of $(0,1)$, that
arise in the two-dimensional case instead of the ``standard bubble'' in higher
dimensions. The paper uses the notion of concentration of \cite{AOT, AT}
developed in the spirit of Solimini \cite{Solimini} and of \cite{acc}.
\par
\noindent  2000  \! {\em Mathematics  Subject  Classification.}
Primary  \! 35J20, 35J35, 35J60; Secondary 46E35, 47J30, 58E05. 
\end{abstract}

\section{Introduction}
Quasilinear elliptic problems in the Sobolev space $W^{1,p}$  with $p=N$ in
dimension $N$ are, in many respects, different from the case of Sobolev
spaces with $N>p$. The counterpart of Sobolev imbeddings in this case is the
Trudinger-Moser inequality (Yudovich, Peetre, Pohozhaev, Trudinger and Moser
\cite{Yudovich,Peetre, Pohozhaev, Trudinger,Moser}). Analysis of problems
involving the corresponding nonlinearity $e^{b |u|^{N'}}$, where
$N'=\frac{N}{N-1}$, $b>0$, often finds no counterpart in properties of the
critical Sobolev nonlinearity $|u|^\frac{pN}{N-p}$ when $N>p$.

This paper deals with properties of critical sequences for semilinear elliptic problem associated with the Trudinger-Moser functional in dimension $2$. Adimurthi and Struwe  proposed in \cite{AdiStr} that every Palais-Smale sequence for the semilinear functional \eqref{J} (and more generally, for similar functionals with nonlinearity of critical growth in the sense of Adimurthi (\cite{Adi-Pisa}), has a convergent subsequence except at the levels $J=\frac{n}{2}$, $n\in\N$.
Bounded critical sequences divergent at these levels were constructed by
Adimurthi and Prashanth \cite{AP}. Similarly to the higher dimensions case with
the critical nonlinearity (that lacks weak continuity), Palais-Smale sequences
may diverge due to concentration phenomena. In the higher-dimensional case,
critical sequences for semilinear elliptic problems with critical nonlinearities
are structured as finite sums of linear blowups of a solution of the asymptotic
equation. The latter is, in many cases, the unique (up to dilations and
translations) positive solution of $-\Delta u=u^\frac{N+2}{N-2}$ in $\R^N$,
known as {\em standard bubble}, {\em instanton} or {\em
(Bliss-)Talenti solution}. Analogous blowup analysis of solution sequences for
$N=2$ yields a counterpart of the standard bubble, but for general critical
sequences the pattern of concentration remained until recently unclear.
Adimurthi-Struwe conjecture was motivated by the Adimurthi's proof of
compactness of critical sequences below the level $1/2$ \cite{Adi-Pisa},
existence of a concentrating Palais-Smale sequence at the level $1/2$ (and thus,
by adding translated sequences with distinct concentration points, at the levels
$n/2$,  $n\in\N$), due to Adimurthi and Prashanth \cite{AP}, and the Druet's
(\cite{Druet}) analysis of blowup for sequences of exact solutions.

Linear blowups (rescalings) do not exhaust all relevant concentration in
the Trudinger-Moser case, that is, no subtraction of blowup terms from a
general sequence will yield a convergent remainder, as it is the case in the
higher dimensions (Solimini \cite{Solimini}, see also Gerard \cite{Gerard} and
Jaffard \cite{Jaffard}). In particular, the  critical
sequence $u_k$ in the Trudinger-Moser setting, given by Adimurthi and
Prashanth \cite{AP} (Theorem A), which concentrates in the sense that $|\nabla
u_k|^2$ converges weakly to the Dirac delta-function, shows no nontrivial weak
limits under linear deflations. For the case of unit disk
the sequence in \cite{AP} is $\lambda_k\mu_{t_k}(r)$, where the functions
$\mu_t$, $t>0$, are the celebrated
Moser functions (see \cite{Moser})
\begin{equation}
\label{Moser}
 \mu_t(r)=\frac{\min\{\rlog{t},\rlog{r}\}}{\sqrt{2\pi\rlog{t}}},
\end{equation}
$t_k\to +0$, and $\lambda_k\in\R$ is a particular sequence convergent to  $1$.
Euclidean deflations of this sequence yield in the
weak limit only constant functions on $\R^2$, which all
represent the zero element of $\mathcal D^{1,2}(\R^2)$, the completion of
$C_0^\infty(\R^2)$ in the gradient norm. (In fact, deflations of functions
\eqref{Moser} are usually invoked to illustrate that  $\mathcal D^{1,2}(\R^2)$
is not a function space because its zero element spans $1$).

Concentration (defined by emergence of a singular part in the weak*
limit of the sequence of measures $|\nabla
u_k|^2\ud x$) cannot be reduced to linearly rescaled profiles in other
applications as well. For example, in subelliptic problems on stratified
nilpotent Lie groups concentration occurs via anisotropic blowups
(\cite{STLie}). It should be noted, however, that once one does not restrict the
class of operators responsible for formation of profile to linear blowups, all
concentration defined in terms of singular limits of Lagrangean densities can be
expressed in terms of a profile decomposition (\cite{acc}).

In the radial Trudinger-Moser  case \cite{AOT}, nonzero profiles occur under
inhomogeneous blowups
\begin{equation}
\label{gs}
\delta_su(r)=s^{-1/2}u(r^s),\,s>0, u\in H^1_{0;r}(B)
\end{equation}
(by $B$ we denote an open unit disk, we will use the notation $B_R$ for an open
disk of radius $R$ and also indicate its center if it is different from the
origin).
The operators $\delta_s$, $s>0$, form a unitary (multiplicative) group $G$ acting 
on $H^1_{0;r}(B)$, and most
remarkably, the set of Moser functions is invariant under $G$. It is shown in \cite{AOT},
following \cite{acc} and in the spirit of \cite{Solimini}, that any bounded
sequence in $H^1_{0;r}(B)$ has a subsequence convergent in $\exp L^2$-norm (the
Orlicz norm of the Trudinger-Moser functional) once one subtracts from it
concentrating terms of the form $\delta_s w$. In particular, the
Adimurthi-Prashanth sequence can be represented as an inhomogeneous blowup
sequence  of a Moser function, $\lambda_k\delta_{s_k}\mu_{1/e}$ with $s_k\to 0$.

In \cite{AT} the concentration analysis of \cite{AOT} is extended to the
non-radial case
(the main results of \cite{AOT} were later reproduced, with an independent
proof, by \cite{BMM}). Remarkably, the inhomogeneous deflations $u(z)\mapsto
n^{-1/2}u(z^n)$ produce radial concentration profiles even for nonradial
sequences. The present paper is based on the profile decomposition of \cite{AT}
and uses it to derive the appropriate asymptotic equations and the
correspondent profiles. It shows than that these profiles can indeed occur in
critical sequences.

The results of this paper are as follows. We define the family of functions
(``Moser-Carleson-Chang towers'') that appear as concentration profiles and
present some of their properties in Section 2.
In Section 3 we establish properties of all concentration profiles that may occur in  critical sequences
(Theorem~\ref{lem:toy_pyramid}). These profiles, denoted as $\mu_{C_+,C_-}$,
are radial functions equal to $\pm\sqrt{\frac{1}{2\pi}\rlog{r}}$ on closed sets
$C_\pm\subset (0,1)$, and are harmonic on $(0,1)\setminus C$, $C=C_+\cup C_-$.
In particular, when $C_+=\{t\}$, $C_- = \emptyset$, the function  $\mu_{C_+,C_-}$
is the Moser function $\mu_t$.
In Section 4 we show that the Adimurthi-Struwe conjecture is false, namely, that for every level $c> \frac12$
there is an infinite collection of closed sets $C\subset (0,1)$ such that for
every such set $C$, there is a critical sequence of the form
$u_k=\delta_{1/s_k}\mu_C+\psi_k$, $s_k\to 0$, $\|\nabla\psi_k\|_2\to 0$ with
$J(u_k)\to c$. A comparable task would be elementary in the higher-dimensional
case, where critical sequences remain critical under perturbations vanishing
in the Sobolev norm, but this is not the case here (or in \cite{AP} as
well). Given that the argument in the general case would overtask even a most 
patient and motivated reader, we focused here on existence of critical
sequences with the concentration profile $\mu_{C_+,C_-}$ only in the cases
when $C_=\emptyset$ and $C_+$ is either a (possibly uncountable) set of measure
zero (Theorem~\ref{thm:critseq}) or an interval (Theorem~\ref{badlevelsC}). The
argument presented there, together with the case of a finite $C$ in
Theorem~\ref{badlevels}, contains all technical points needed for the general
case. Each of these cases suffices for the negative answer to the
Adimurthi-Struwe conjecture.

We conclude the paper with Theorem~ \ref{thm:asymp} that contains a general profile decomposition for critical sequences of the Trudinger-Moser functional.

\section{Definitions and assumptions}
\begin{definition}
\label{Mosertower}
{\bf (Moser-Carleson-Chang tower functions)}
Let $C_+, C_-$ be closed subsets of $(0,1)$, such that $C=C_+\cup
C_-\neq\emptyset$,  let $A=(0,1)\setminus C$,  and let
$\mathcal A=\{(a_n,b_n)\}_n$ be an enumeration of all connected components of $A$ starting with $a_1=0$.
A continuous radial function $\mu_{C_+,C_-}\in H^1_{0}(B)$ is called a  {\em Moser-Carleson-Chang tower} if
\begin{equation}
 \label{w}
\mu_{C_+,C_-}(r)=\begin{cases}
            \sqrt{\frac{1}{2\pi}\rlog{r}}, & r\in C_+,\\
 -\sqrt{\frac{1}{2\pi}\rlog{r}}, & r\in C_-,\\
 A_n +B_n\rlog{r}, & r\in(a_n,b_n),  A_n,B_n\in\R.
           \end{cases}
\end{equation}
If $C_-=\emptyset$, we will use the notation $\mu_C$ instead.
\end{definition}

When the set $C_+$ consists of a single point $t\in(0,1)$ and $C_-=\emptyset$,
the function $\mu_{C}$ is the original
Moser function \eqref{Moser}. When $C\subset (0,1)$ is a closed interval,
a function of the form $\mu_C$ was found
in the proof of existence of extremals for the  Trudinger-Moser functional
by Carleson and Chang, \cite{CarlesonChang}, p. 121, written in the
variable $t=\rlog{r}$).
Let us prove some elementary properties of Moser-Carleson-Chang towers.
\begin{proposition}
 \label{lem:LM}
 \par
 \begin{description}
\item[(i)] The coefficients $A_n$, $B_n$ are defined uniquely by continuity at $a_n, b_n\in C$. In particular, if $C=C_+$,
\begin{equation}
\label{LM}
\begin{split}
A_n= \frac{1}{\sqrt{2\pi}}\frac{\sqrt{\rlog{a_n}}\sqrt{\rlog{b_n}}}{\sqrt{\rlog{a_n}}+\sqrt{\rlog{b_n}}},\\
B_n=\frac{1}{\sqrt{2\pi}}\frac{1}{\sqrt{\rlog{a_n}}+\sqrt{\rlog{b_n}}}
\end{split}
\end{equation}
(when  $n=1$ the values  in \eqref{LM} are understood in the sense of the limits
as $a_1\to 0$,
i.e. $A_1=\sqrt{\frac{1}{2\pi}\rlog{b_1}}$ and $B_1=0$).
\item[(ii)] The function $\mu_{C_+,C_-}(r)$ has continuous derivative at every
point of $(0,1)$ except
$\{a_n, b_n\}_{(a_n, b_n)\in \mathcal A}$.
\item[(iii)] Let $\mathcal A'$ be the set of all intervals $(a,b)\in\mathcal A$ where $\mu_{C_+,C_-}$ does not change sign, and let
${\mathcal A}'' = {\mathcal A}\setminus {\mathcal A}^\prime$. Then
\begin{equation}
\label{sigmas}
\|\nabla \mu_{C_+,C_-}\|^2_2=\frac14\int_C\frac{\ud r}{r\rlog{r}}+\sum_{(a,b)\in\mathcal A'} \frac{\sqrt{\rlog{a}}-\sqrt{\rlog{b}}}{\sqrt{\rlog{a}}+\sqrt{\rlog{b}}}+\sum_{(a,b)\in\mathcal A^{\prime\prime}} \frac{\sqrt{\rlog{a}}+\sqrt{\rlog{b}}}{\sqrt{\rlog{a}}-\sqrt{\rlog{b}}}.
\end{equation}
\item[(iv)]The number of zeroes of $\mu_{C_+,C_-}$ on $(0,1)$ is less than the
value of $\|\nabla \mu_{C_+,C_-}\|_2^2-1$.
\item[(v)] For any choice of $C_-,C_+$, one has $\|\nabla \mu_{C_+,C_-}\|_2^2\ge
1$ and the equality holds only if $C$ consists of one point.
\end{description}
 \end{proposition}
\begin{proof}
(i): Values \eqref{LM}  for $n\ge 2$ are the unique solutions of continuity conditions at $a_n$ and $b_n$, $A_n+B_n\rlog{a_n}=\sqrt{\frac{1}{2\pi}\rlog{a_n}}$ and
$A_n+B_n\rlog{b_n}=\sqrt{\frac{1}{2\pi}\rlog{b_n}}$. Since $\mu_{C_+,C_-}$ has
a finite Sobolev norm, we have, necessarily, $B_1=0$, which yields
$A_1=\sqrt{\frac{1}{2\pi}\rlog{b_1}}$.

(ii): For the sake of simplicity we consider the case $C=C_+$, the general case is similar.
If $(a_{n_k},b_{n_k})\subset\mathcal A$ and $a_{n_k}\to c$ for some $c$, then necessarily $b_{n_k}\to c$, from which, by elementary computation, follows
$\mu_C'(c)=\lim\mu_C'(a_{n_k})=\lim\mu_C'(b_{n_k})=\left(\sqrt{\frac{1}{2\pi}\rlog{r}}\right)'|_{r=c}$.
Consequently, since $\mu_C$  by definition is smooth at all internal points of $A$ and of $C$,
the only points in $(0,1)$ where $\mu_C'$ is discontinuous are the points $a_n$ and $b_n$.

(iii) follows from the direct computation of the right hand side in
\[
\|\nabla \mu_{C_+,C_-}\|_2^2=2\pi\int_{C_+,C_-}|\mu'_{C_+,C_-}(r)|^2r\ud
r+2\pi\sum_n \int_{a_n}^{b_n}|\mu_{C_+,C_-}'(r)|^2r\ud r.
\]

(iv): The terms in \eqref{sigmas} corresponding to the set $\mathcal A^{\prime\prime}$ are greater than $1$.
Furthermore, on the interval $(a,1)\in\mathcal A$, one has necessarily $\mu_{C_+,C_-}(r)=\pm\frac{\rlog{r}}{\sqrt{2\pi\rlog{a}}}$ and
the contribution of this interval to  \eqref{sigmas} is
\[
\int_{r\in(a,1)}|\nabla \mu_{C_+,C_-}|^2=1.
\]
(v): By the last observation the sum in \eqref{sigmas} is greater or than $1$,
and the equality occurs only if the sum consists only of this term, which
corresponds to $\mathcal A=\{(0,a),(a,1)\}$.
\end{proof}

We consider in this paper critical nonlinearities following the definition from \cite{Adi-Pisa}. Without loss of generality we restrict the consideration to the factor $b$ in the exponent equal to $4\pi$, since the general case can be always recovered by replacing the variable $u$ with a scalar multiple. Let $f\in C(\R)$ and $F(s)=\int_0^sf(t)\mathrm{d}t$.

\begin{definition} We say that a continuous function $f:\R\to\R$ is of the $4\pi$-critical growth, if
$f(s)=g(s)e^{4\pi s^2}$ and for any $\delta>0$,
\[
\lim_{|s|\to\infty }g(s)e^{-\delta s^2}=0.
\]
\end{definition}
We will study the functional

\begin{equation}
\label{Jv}
J(u)=\frac12\int_\Omega |\nabla u|^2\ud x-\frac{1}{8\pi}\int_\Omega F(u)\ud x, \;u\in H_0^1(\Omega),
\end{equation}
where $\Omega$ is a bounded domain in $\R^2$.

We write $g(t)=\frac{1}{8\pi}h'(t)+h(t)t$ so that
$\frac{1}{8\pi}f(t)=g(t)e^{4\pi t^2}$, and we will use the following assumptions:
\begin{description}
\item[(g0)] $\lim_{|s|\to\infty }\frac{g'(t)}{g(t)t}=0$;
\item[(g1)]  There is a $T>0$ such that $\inf_{t\ge T}g(t)>0$ and $\sup_{t\le -T}g(t)<0$;
\item[(g2)] $\lim_{|t|\to \infty}\frac{F(t)}{f(t)t}=0$.
\end{description}

\begin{remark}
Examples of $g(t)$ can be found in \cite{Adi-Pisa}. In particular, $g(t)=t$ is a typical example.
\end{remark}

\section{Blow-up profiles}
The profile decomposition in $H_0^1(\Omega)$ below is quoted from \cite{AT},
Theorem~2.5 combined with Theorem~2.6.
The notation $z^j$, $j\in\N$, $z\in\Omega$, is understood in the sense of the power of a complex number representing the point $z$.  Without loss of generality we assume that $\Omega\subset B_\frac12$.
Functions are in $H_0^1(\Omega)$ are considered also as elements of  $H_0^1(B)$, via extension by zero.
\begin{theorem}
\label{2cc-j} Let $\Omega\subset\R^2$ be a bounded domain and let
$u_k$ be a bounded sequence  in $H^{1}_{0}(\Omega)$. There
exist $j_k^{(n)}\in \N$, with $j_k^{(1)}=1$,  and
$z_k^{(n)}\in \bar \Omega$, with $z_k^{(1)}=0$ and
$\lim_{k\to\infty} z_k^{(n)}= z_n\in\bar\Omega$, $k \in\N$, $n\in\N$,
such that  for a renumbered subsequence,
\begin{eqnarray}
\label{w_n=j} &&w^{(n)}(|z|)=\wlim \left({j_k^{(n)}}\right)^{-1/2}u_k(z_k^{(n)}+z^{j_k^{(n)}}),
\\
\label{separates=j} &&    z_m\neq z_n \mbox{ or } |\log j_k^{(m)}-\log j_k^{(n)}|\to \infty  \mbox{ whenever } n \neq m,
\\
\label{norms=j} &&\sum_{n \in \N} \int_B|\nabla w ^{(n)}|^2\dx  \le \limsup \int_\Omega|\nabla u_k|^2\dx,
\\
\label{BBasymptotics=j} &&u_{k}  - \sum_{n\in\N}  {{j_k^{(n)}}}^{1/2}w^{(n)}(|z-z_n|^{1/j_k^{(n)}}) \to 0 \text { in }\exp L^2,
\end{eqnarray}
and the series  $\sum_{n\in\N}  {{j_k^{(n)}}}^{1/2}w^{(n)}(|z-z_n|^{1/j_k^{(n)}})$ converges in $H^1_{0}(B)$
uniformly in $k$.
\end{theorem}

We use the fact that $u_k$ is a critical sequence for \eqref{J} in order to make the expansion
\eqref{BBasymptotics=j} more specific, namely, to verify that every asymptotic profile \eqref{w_n=j} is a Moser-Carleson-Chang tower and that the expansion \eqref{BBasymptotics=j} has finitely many terms. This is stated at the end of the paper as Theorem~\ref{thm:asymp}.

\begin{theorem}
\label{lem:toy_pyramid} Assume that the function $f$ is of $4\pi$-critical growth and satisfies (g0) and (g1).
Let $u_k$ be a critical sequence of \eqref{Jv}. Then every concentration profile $w^{(n)}$, $n\ge 2$, given by \eqref{w_n=j},
equals a function $\mu_{C^{(n)}_+,C^{(n)_-}}$ with some disjoint closed sets $C^{(n)}_+,C^{(n)}_-\subset (0,1)$, as given by Definition~\ref{Mosertower}.
\end{theorem}
\begin{proof}
Let us derive first the equation satisfied by the limit \eqref{w_n=j}. The index $n$ is fixed for the rest of the proof, and will be omitted.
Let $\varphi\in C_0^\infty(B)$ be a radial function and let $\psi\in C^1(S^1)$. Evaluating $J'(u_k)$ in the direction $v_k(z)= j_k^{1/2}\varphi(r^{1/j_k})\psi(\theta)$, written in the polar coordinates $z=z_k+re^{i\theta}$,
we get $\langle J'(u_k),v_k\rangle$ in the following form:
and setting $r^{1/j_k}=\rho$, we arrive at
\begin{equation}
\label{2.2}
\begin{split}
\int_0^1\int_0^{2\pi} w'(\rho)\varphi'(\rho)\psi(\theta)\rho\ud\rho\ud\theta\, -
\\
- \int_0^1\int_0^{2\pi} j_k^{3/2} \rho^{2j_k-2}
\frac{1}{8\pi}f(u_k(\rho^{j_k},\theta)) \varphi(\rho)\psi(\theta) \rho\ud\rho\ud\theta\to 0.
\end{split}
\end{equation}
This implies that
\begin{equation}
\label{2.4}
j_k^{3/2} \rho^{2j_k-2}
\frac{1}{8\pi}f(u_k(\rho^{j_k},\theta))
\rightharpoonup -\Delta w \text{ in } H^{-1}(B).
\end{equation}
Recall that $\frac{1}{8\pi}f(s)=g(s)e^{4\pi s^2}$. It easily follows from (g0) and (g1) that
\begin{equation}
 \label{f1}
\lim_{|s|\to\infty}\frac{\log g(s)}{s^2}=0.
\end{equation}
Note that $|w(r)|\le \sqrt{\frac{1}{2\pi}\logfrac}$ for all $r\in(0,1]$. Indeed, if for some $a\in (0,1]$ a converse inequality is true, then, for all $k$ sufficiently large, $4\pi u_k(\rho^{j_k},\theta)^2-2j_k\rlog{\rho}$ will be bounded away from zero when $r$ is in some neighborhood of $a$ and $\theta\in S^1$, and thus, taking into account \eqref{f1}, we have the left hand side in \eqref{2.4} uniformly convergent to $\infty$ on an interval. Taking a positive test function supported on an interval, we arrive at a contradiction, since  $-\Delta w$ is a distribution.

Let $C_1=\left\lbrace r\in(0,1]: |w(r)|=\sqrt{\frac{1}{2\pi}\logfrac}\right\rbrace$. Since $w$ is continuous on $(0,1]$, the set $C_1$ is relatively closed in $(0,1]$. Since $w\in\sob$ and $\sqrt{\frac{1}{2\pi}\logfrac}\notin\sob$, $C_1\neq (0,1]$.
Thus the complement of $C_1$ in $(0,1]$ is an at most countable union of open  intervals of $\R^N$. Let $\mathcal A$ be an enumeration of all such intervals. If $(a,b)\in \mathcal A$, then
$w(a)=\pm\sqrt{\frac{1}{2\pi}\log\frac{1}{a}}$, $w(b)=\pm\sqrt{\frac{1}{2\pi}\log\frac{1}{b}}$
and $|w(r)|<\sqrt{\frac{1}{2\pi}\logfrac}$ for $r\in(a,b)$. From \eqref{2.4} it follows that $w$ is harmonic on $(a,b)$, and, as a radial function, has the form $A+B\rlog{r}$, $A,B\in\R$, and the values of $A$ and $B$ are uniquely defined by the values $w(a)$ and $w(b)$.
Let $C=C_1\setminus \{1\}$. It remains to show that $w$ is constant in a
neighborhood of zero and is harmonic in a neighborhood of $1$.
Assume first that $\mathcal A$ is infinite, and thus countable in this case. 
Then, by elementary calculations already mentioned in the proof of Proposition~\ref{lem:LM}, we have that $w$ satisfies \eqref{sigmas}, which in turn shows that the set $\mathcal A^{\prime\prime}$ of intervals in $\mathcal A$, where the function $w$ changes sign, is finite.
Setting $\sigma_n=1$ when $a_n =0$, $\sigma_n=+\infty$ when $b_n=1$ and
$$
\sigma_n\eqdef\dfrac{\sqrt{\rlog{a_n}}}{\sqrt{\rlog{b_n}}} \text{ otherwise,}
$$
we have
\begin{equation}
 \label{eq:not2}
\|\nabla w\|_2^2\ge \sum_{(a_n,b_n)\in \mathcal A'}\frac{\sigma_n-1}{\sigma_n+1}.
\end{equation}

Note that the sequence $\sigma_n$ is bounded,  otherwise the sum above would have infinitely many terms greater than $1/2$, say $\sigma_n\le M-1$, $M>0$. Then from the relation above it is immediate
that $\sigma_n\to 1$, and
$$
\prod_n \sigma_n\le C \prod_{(a_n,b_n)\in \mathcal A'} \sigma_n\le C e^{\sum_n(\sigma_n-1)}\le C e^{M\sum_n\frac{\sigma_n-1}{\sigma_n+1}}\le
C e^{M\|\nabla w\|_2^2}= \hat{C} <\infty.
$$
Let $\nu$ be any finite subset of $\mathbb Z$ such that $(a_n, b_n)_{n\in \nu}$ are ordered by $n$ and none of $a_n$ is zero. Then
\[
\dfrac{\max_{n\in\nu} \sqrt{\rlog{a_n}}}{\min_{n\in\nu} \sqrt{\rlog{b_n}}}
\le \prod_{n\in\nu} \sigma_n\le \prod_{n\in\mathbb Z} \sigma_n\le \hat{C},
\]
from which one immediately concludes that there are no sequences $(a_{n_k},b_{n_k})\in \mathcal A$ such that $a_n>0$ and
$a_{n_k}\to 0$ or $b_{n_k}\to 1$.

If the set $\mathcal A$ is finite, this is obviously is the case as well. Thus there exist an $\epsilon>0$ such that
the function $w$ is harmonic on the whole interval $(0,\epsilon)$ resp.
$(\epsilon,1)$, unless it equals $\pm\sqrt{\frac{1}{2\pi}\logfrac}$
on this interval. The latter, however, cannot occur since this contradicts
$w\in H_0^1(B)$.
\end{proof}

\section{Critical sequences: the case of finite $C$}
In this and the next section we assume that $\Omega$ is the open unit disk $B$.
\begin{lemma}
\label{norm-of-tower}
Let $0<a_1<\dots<a_n<1$, $n\in\N$, and let $C_+=\{a_1,\dots,a_n\}$, $C_-=\emptyset$. Then
\[
\begin{cases}
\|\nabla \mu_C\|_2^2=1, &  n=1\\
1<\|\nabla \mu_C\|_2^2<n, & n=2,3,\dots
\end{cases}
\]
Furthermore, for any $t\in(1,n)$ there exist $a_1,\dots,a_n $, $0<a_1<\dots<a_n<1$ such that  $\|\nabla \mu_C\|_2^2=t$.
\end{lemma}
\begin{proof}
It follows from \eqref{sigmas} with $\mathcal A=\{(0,a_1), (a_1,a_2),\dots, (a_n,1)\}$
\begin{equation}
\label{eq:not3}
\|\nabla \mu_C\|_2^2=\sum_{j=0}^{n}\frac{\sigma_j-1}{\sigma_j+1}.
\end{equation}
Note that $\sigma_n=+\infty$, $\sigma_0=1$, the first term in the sum equals
zero, the last term equals $1$, and all the intermediate terms have values in
the interval $(0,1)$. This proves the first assertion of the lemma.
To prove the second assertion, let the number $\gamma$ satisfy the equality $\frac{\gamma-1}{\gamma+1}=\frac{t-1}{n-1}$ and note that $\gamma>1$.
Let $\sigma_2=\dots=\sigma_{n-1}=\gamma$.
Then, since $\sigma_n=+\infty$ and $\sigma_0=1$, we have $\sum_{j=0}^{n}\frac{\sigma_j-1}{\sigma_j+1}=t$ and,
since $\gamma=\sigma_j=\sqrt{\frac{\rlog{a_{j}}}{\rlog{a_{j+1}}}}$, $j=2,\dots,n$, we have $\alpha_{j}=\gamma^2 \alpha_{j+1}$, where $\alpha_j = \rlog{a_{j}}$, which recursively
defines suitable $a_1,\dots,a_{n-1}$ once we arbitratily set the value of $a_n\in(0,1)$. \end{proof}

The following statement is a corollary of Lemma~2.1 from \cite{AP}.
\begin{lemma}
 \label{ap3}
 Let let
$J$ be the functional \eqref{Jv} with the function $f$ of critical growth.
Let $w\in H_{0;r}^1(B)$, $v_k\to 0$ in $H_{0;r}^1(B)$
and $s_k\to\infty$. If
\begin{equation}
 \label{ap3b}
 \int_0^rs_k^{3/2}\frac{f(s_k^{1/2}(w(\rho)+v_k(\rho)))}{8\pi}
 \rho^{2s_k-2}\rho \ud \rho\to -rw'(r)=-\int_0^r \Delta w(\rho)\rho\ud\rho
\text{ in } L^2(B),
\end{equation}
then the sequence $u_k=\delta_{1/s_k}(w+v_k)$ is critical for the functional $J$.
\end{lemma}
\begin{proof} From $v_k\to 0$ and \eqref{ap3b} we have
\begin{equation}
 \label{ap2}
 \int_0^1\left|
 \int_0^r\left(\Delta w(\rho)+s_k^{3/2}\frac{f(s_k^{1/2}w_k(\rho))}{8\pi}
 \rho^{2s_k-2}\right)\rho \ud \rho\right|^2
 r\ud r\to 0,
\end{equation}
which immediately implies
\begin{equation}
 \label{ap}
 \int_0^1\left|rw_k'(r)+
 \int_0^rs_k^{3/2}\frac{f(s_k^{1/2}w_k(\rho))}{8\pi}
 \rho^{2s_k-1}\ud \rho\right|^2
 r^{2s_k-1}\ud r\to 0.
\end{equation}
Rewrtiting \eqref{ap} in variables $\tilde r= r^{s_k}$, $\tilde \rho= \rho^{s_k}$, we have
\begin{equation}
 \label{ap0}
 \int_0^1\left|ru_k'(r)+
 \int_0^r\frac{f(u_k(\rho))}{8\pi}
 \rho\ud \rho\right|^2
 r\ud r\to 0,
\end{equation}
which is the condition of Lemma~2.1 in \cite{AP} that provides that $u_k$ is a
critical sequence.
\end{proof}
 \begin{theorem}
\label{badlevels} Let $C_-=\emptyset$, $C_+=\{a_1,\dots,a_n\}$,  $0<a_1<a_2<\dots<a_{n}<1$, $n\in\N$, and let
$J$ be the functional \eqref{Jv} with the function $f\in C^1$ of $4\pi$-critical growth satisfying (g0), (g1) and (g2).
Then there exist sequences $s_k\to\infty$, $v_k\in C_{0;r}^1(B)$ ,
$\|v_k'\|_\infty\to 0$,
such that the sequence $u_k=\delta_{1/s_k}(\mu_{C}+v_k)$ satisfies
\begin{equation}
\label{crit1}
J(u_k)\to \frac{1}{2}\|\mu_C\|_2^2, J'(u_k)\to 0 {\in} H_0^1(\Omega).
\end{equation}
Furthermore, for every $c\in(\frac12,\frac{n}{2})$ there exist points $0<a_1<a_2<\dots<a_{n}<1$ such that $\frac{1}{2}\|\mu_C\|_2^2=c$.
\end{theorem}
\begin{proof}
Let us prove first that $J'(u_k)\to 0$.
By Lemma~\ref{ap3}, it suffices to find a radial vanishing sequence $v_k$ such that
$u_k=\delta_{1/s_k}(\mu_{C}+v_k)$ satisfies \eqref{ap3b}.
Let us require first \eqref{ap3b} with only weak convergence, namely,
\begin{equation}
\label{weaktow}
I_k(\varphi)=\int_0^1s_k^{3/2}\frac{f(s_k^{1/2}(\mu_C(r)+v_k(r))}{8\pi}r^{2s_k-2}\varphi(r)r\ud r \rightarrow \int_0^1\mu_C'(r)\varphi'(r)r\ud r
\end{equation}
for every $\varphi\in C_{0;r}^\infty(B)$.

An elementary calculation of the right hand side gives
\begin{equation}
\label{deltaw}
\int_0^1\mu_C'(r)\varphi'(r)r\ud r=\sum_{j=1}^n q_j\varphi(a_j),
\end{equation}
with some $q_i\in\R$. Specific values of $q_i$ are not invoked in the subsequent argument, but in order that the reader will not feel empty-handed,
we quote them nonetheless:
\begin{equation}
\label{qj}
q_j=\frac{1}{\sqrt{2\pi}}\frac{\alpha_{j-1}-\alpha_{j+1}}{(\alpha_{j}+\alpha_{j+1})(\alpha_{j}+\alpha_{j-1})}, \;j=1,\dots,n,
\end{equation}
which for $j=1$ is to be understood in the sense of the limit as $\alpha_0\to+\infty$, and where, as before, $\alpha_j=\sqrt{\rlog{a_j}}$.

In order to evaluate $I_k(\varphi)$, we change the integration variable to $t=\rlog{r}$:
\begin{equation}
\label{Ik}
I_k(\varphi)=\int_0^\infty e^{s_k(4\pi \mu_C(e^{-t})^2-2t+\psi_k(t))}\varphi(e^{-t})\ud t,
\end{equation}
where
\begin{equation}
\label{psik}
\psi_k(t)=\frac32\frac{\log s_k}{s_k}+\frac{\log g(s_k^{1/2}(\mu_C(e^{-t})+w_k(e^{-t}))}{s_k}+8\pi \mu_Cv_k+4\pi v_k^2.
\end{equation}
Evaluation of $4\pi \mu_C(e^{-t})^2-2t$ on the interval $t\in (\alpha_{j+1}^2,\alpha_{j}^2)$, corresponding to
$r\in (a_j,a_{j+1})$ gives, after elementary computation and recalling the notation $\sigma_j=\frac{\alpha_j}{\alpha_{j+1}}$, the two following identities:
\[
4\pi \mu_C(e^{-t})^2-2t=-2\frac{\sigma^2_j-1}{\sigma^2_j+1}(\alpha_j^2-t)+4\pi B_j^2(t-\alpha_j^2)^2,\,t\in (\alpha_{j+1}^2,\alpha_{j}^2),
\]
and (after replacing the index $j$ with $j-1$)
\[
4\pi \mu_C(e^{-t})^2-2t=2\frac{\sigma^2_{j-1}-1}{\sigma^2_{j-1}+1}(t-\alpha_j^2)+4\pi B_{j-1}^2(t-\alpha_j^2)^2, t\in (\alpha_{j}^2,\alpha_{j-1}^2).
\]
Let $\gamma_j^-=\left(2\frac{\sigma^2_{j-1}-1}{\sigma^2_{j-1}+1}\right)^{-1}$ and $\gamma_j^+=\left(2\frac{\sigma^2_{j}-1}{\sigma^2_{j}+1}\right)^{-1}$.
For every $j=1,\dots,n$, set
\[
 M_j=(\gamma_j^-+\gamma_j^+)q_j
\]
and consider the equation
\begin{equation}
\label{v}
8\pi \mu_C(r)v_k(r)+4\pi v_k(r)^2+\frac{\log g(s_k^{1/2}(\mu_C+v_k))}{s_k}+\frac{\log s_k}{2s_k}=\frac{\log M_j}{s_k}, r\in(0,1).
\end{equation}
Fix $\epsilon\in (0,\max\{a_1,1-a_n\})$.
For all $k$ large enough, by the implicit function theorem, using (g0) and the assumption that $g\in C^1$,
we have a unique solution $v_k^{(j)}$ that converges to zero uniformly on $[\epsilon,1-\epsilon]$.
Note that from \eqref{v} follows that, with $\psi_k^{(j)}$ as in \eqref{psik} with  $v_k=v_k^{(j)}$,
\begin{equation}
\label{psik2}
e^{s_k\psi_k^{(j)}}=M_j s_k.
\end{equation}

Differentiation of \eqref{v} with respect to $r$ together with (g0) implies that $v_k^{(j)\prime}$ converges to zero uniformly on $[\epsilon,1-\epsilon]$.
Let $\tilde\omega_j\ni a_j$, $j=1,\dots,n$ be disjoint subintervals of $(0,1)$, $\omega_j\ni a_j$ be closed subintervals of $\tilde\omega_j$ and let
$\{\chi_j\}_{j=1,\dots_{n}}\in C_0^\infty(0,1)$ be supported in
$\tilde\omega_{j}$ and equal $1$ on  $\omega_j$. Define
$v_k=\sum_{k=1}^n\chi_jv_k^{(j)}$. Then $v_k\to 0$ in $H_0^1(B)$.
Furthermore, lengthly but elementary calculations show that $v_k$ also satisfies \eqref{weaktow}. We give here only a sketch of the calculations,
leaving details to the reader. The principal term in the exponent of \eqref{Ik} equals, thanks to \eqref{psik2},
\[
s_k(4\pi \mu_C(e^{-t})^2-2t)=
\begin{cases}
-\frac{s_k}{\gamma_j^+}(t-\alpha_j^2)+O(|\alpha_j^2-t|^2), & t\ge \alpha_j^2,
\\
-\frac{s_k}{\gamma_j^-}(\alpha_j^2-t)+O(|\alpha_j^2-t|^2), & t \le \alpha_j^2,
\end{cases}
\]
whch tends to $-\infty$ uniformly outside of arbitrarily small neighborhoods of $\alpha_j$. This assures that all the nonzero contributions to the limit come from small neighborhoods of points $\alpha_j^2$ (corresponding to $a_j$). The exact value of the exponent in a neighborhood of $\alpha_j^2$ is set by \eqref{v}, so that the elementary integration in the variable $\tau=s_k(t-\alpha_j^2)$ over $-\infty<\tau<0)$ and over $0<\tau<\infty$
gives in the limit a multiple of $\varphi(r-a_j)$, which matches coefficient \eqref{qj} by the choice of the parameter $M_j$ above.
\par
Then \eqref{ap3b} is immediate since the sequence in \eqref{weaktow} converges 
uniformly to zero outside of the points $a_j$ and thus \eqref{ap3b}
holds true, and the integration is reduced to neighborhoods of $a_j$, where the
argument above can be repeated with only trivial modifications. This completes
the proof for $J'(u_k)\to 0$.

From (g2) it follows immediately that $\int F(u_k)\to 0$, and thus
\begin{equation}
 \begin{split}
J(u_k)=\frac12\|\nabla u_k\|_2^2+o(1)=\frac12\|\nabla \delta_{s_k}\mu_C\|_2^2+o(1)\\
=  \frac12\|\nabla \mu_C\|_2^2+o(1).
 \end{split}
\end{equation}
Then the last assertion of the theorem follows from Lemma~\ref{norm-of-tower}.
\end{proof}

\begin{remark}
\par
\begin{enumerate}
 \item Note that since our nonhomogeneous dilations are isometric operators on the Sobolev space,
$u_k-g_{s_k}\mu_{C_n}\to 0$ in $H_0^1(B)$.
 \item The case $n=1$ of Theorem~\ref{badlevels}, corresponding to $c=\frac12$,  is proved in \cite{AP}.
An elegantly balanceed calculation in \cite{AP} allows to pick up the sequence
$v_k$ not merely vanishing in Sobolev norm,
but having an explicit form $\lambda_k\mu_{a_1}$ with $\lambda_k\to 0$.
\end{enumerate}
\end{remark}

\section{Critical sequences: infinite $C$}
\begin{theorem}
 \label{thm:critseq} Let $J$ be the functional \eqref{Jv} with the $4\pi$-critical growth function $f$ satisfying (g0), (g1) and (g2).
 If $C_0=\emptyset$ and $C_+$ is a compact subset of $(0,1)$ of measure zero, then there exists a sequence $s_k\to +\infty$, a sequence
 $\psi_k\in H_{0;r }^1(B)$ such that $\|\nabla \psi_k\|\to 0$, $J'(\delta_{s_k}\mu_C+\psi_k)\to 0$ and
 $J(\delta_{s_k}\mu_C+\psi_k)\to \frac12\|\nabla\mu_C\|_2^2$.
\end{theorem}
\begin{proof}
1. Let $\mathcal A_\epsilon$ be a collection of all maximal intervals contained in $(0,1)\setminus C$, such that that $|A_\epsilon|\ge 1-\epsilon$, where $A_\epsilon=\bigcup_{(a,b)\in \mathcal A_\epsilon}(a,b)$ and $\epsilon>0$ is small enough that both intervals with $0$ or with $1$ as an endpoint are included in $\mathcal A_\epsilon$. Let $B_\epsilon=(0,1)\setminus \bar{A_\epsilon}$ and let $\mathcal B_\epsilon$ be the collection of maximal intervals contained in $B_\epsilon$. Define a finite set $C_\epsilon=\partial A_\epsilon$ in $(0,1)$.  Then, invoking notations from the proof of Theorem~\ref{lem:toy_pyramid}, using \eqref{eq:not3} and noting that \eqref{eq:not3} also remains valid as a positive series for countable collections of intervals, we have
\begin{equation}
\begin{split}
 \|\nabla(\mu_C-\mu_{C_\epsilon}\|_2^2\le 2\pi\int_{B_\epsilon}2(|\nabla \mu_C|^2+|\nabla \mu_{C_\epsilon}|^2)r\ud r
 \\
 \le \sum_{\mathcal A\setminus \mathcal A_\epsilon}(\sigma_i-1) + \sum_{\mathcal B_\epsilon}(\sigma_i-1)
 \le 2 \sum_{\mathcal A\setminus \mathcal A_\epsilon}(\sigma_i-1)
 \\
 \le \frac{2}{\sqrt{\rlog{\sup C}}}
 \sum_{\mathcal A\setminus \mathcal A_\epsilon}(\sqrt{\rlog{a_i}}-\sqrt{\rlog{b_i}}):=\kappa_\epsilon.
\end{split}
 \end{equation}
Note now that, since $|A\setminus A_\epsilon|\to 0$ as $\epsilon\to 0$, the measure of the image of $A\setminus A_\epsilon\subset[\epsilon,1-\epsilon]$ under the map $r\mapsto \sqrt{\rlog{r}}$ vanishes as well and thus  $\kappa_\epsilon\to 0$.
We conclude that for any closed set $C\subset(0,1)$, there is a sequence $C_j$ of finite subsets of $C$ such that $\mu_{C_j}\to\mu_C$
in $H_0^1(B)$.

2. Fix now a sequence $\epsilon_j\to 0$, set $C_j=C_{\epsilon_j}$ and, for every $j\in\N$, consider a sequence
$v_k^{(j)}\in C_0^1(B)$, $s_k^{(j)}\to\infty$, given by Theorem~\ref{badlevels} for the finite set $C_j$, such that $\|v_k^{(j)'}\|_\infty\le 1/k$ and
$\|J'(\delta_{s_k^{(j)}}(\mu_{C_j}+v_k^{(j)}))\|\le 1/k$. Then, for the diagonal sequence with $k=j$ we have
\[
 J'(\delta_{s_j}(\mu_C+\mu_{C_j}-\mu_C +v_j^{(j)}))\to 0.
\]
Let $\psi_j=\delta_{s_j}(\mu_{C_j}-\mu_C +v_j^{(j)})$. Then
\[
 \|\nabla \psi_j\|_2= \|\nabla(\mu_{C_j}-\mu_C +v_j^{(j)})\|_2\le \|\nabla (\mu_{C_j}-\mu_C)\|_2+1/j\to 0.
\]
Since, as we obtained above by diagonalization,  $J'(\delta_{s_j}\mu_C+\psi_j)\to 0$ and, by (g2), $\int_\Omega F(\delta_{s_j}\mu_C+\psi_j)\to 0$,
we have
\[
J(\delta_{s_j}\mu_C+\psi_j))= \frac12\|\nabla(\delta_{s_j}\mu_{C}+\psi_j)\|^2_2+o(1)=\frac12\|\nabla(\mu_{C}\|^2_2+o(1).
\]
\end{proof}

While Theorem~\ref{badlevels} is a relatively natural generalization of
Theorem~A in \cite{AP}, which constructs a critical sequence with the Moser
function as a blowup profile, where the set $C$ is a singleton, and
Theorem~\ref{thm:critseq} follows from Theorem~\ref{badlevels} by an
approximation argument, this gives no insight if there is also a critical
sequence that has a concentration profile $\mu_C$ with the set $C$ of positive
measure. This requires a different construction of the vanishing correction
$v_k$.

\begin{theorem}
\label{badlevelsC} Let $J$ be the functional \eqref{Jv} with the function $f$ of $4\pi$-critical growth satisfying (g0), (g1) and (g2).
Let $c\ge\frac12$ , $\beta=\beta(c)= {e^{8(c-1/2)}}$,  $a\in(0,1)$.
Then there exist sequences $s_k\to\infty$, $\psi_k\in C_{0;r}^1(B)$ , $\|\nabla
\psi_k\|_2^2\to 0$,
 such that the sequence $u_k=\delta_{s_k}\mu_{[a^\beta ,a]}+\psi_k$ satisfies
\begin{equation}
\label{crit1C}
J(u_k)\to c, J'(u_k)\to 0.
\end{equation}
 \end{theorem}
\begin{proof}
1. Let us construct a sequence $v_k$ satisfying \eqref{weaktow}.
An elementary calculation of the right hand  side of \eqref{weaktow} with $C=[a^\beta,a]$ gives
\begin{equation}
\label{deltawC}
\int_0^1\mu_C'(r)\varphi'(r)r\ud r=p\varphi(a^\beta)+q\varphi(a)+
\int_{a^\beta}^a\frac{\mu_C(r)}{4r^2(\rlog{r})^2}\varphi(r)r\ud r,
\end{equation}
with some positive coefficients $p,q$ whose specific values are not important for the construction.
We construct first a sequence that satisfies \eqref{weaktow} with test functions $\varphi$ supported only on $[a^\beta,a]$.
Consider the equation
\begin{equation}
\begin{split}
\frac32 \log s_k-4\pi s_k\mu_C(r)^2+4\pi s_k(\mu_C(r)+v_k(r))^2+\log g(s_k^{1/2}\mu_C(r)+s_k^{1/2}v_k(r))\\
=-\log 4\sqrt{2\pi}-\frac32\log \rlog{r}, \;r\in(a^\beta,a).
\end{split}
\end{equation}
Expanding the square in the third term and dividing the equation by $8\pi s_k
\mu_C(r)$, we have, equivalently,
\begin{equation}
\label{inside}
\begin{split}
v_k(r)+\frac{1}{2\mu_C(r)}v_k(r)^2+\frac{\log
g(s_k^{1/2}\mu_C(r)+s_k^{1/2}v_k(r))}{8\pi s_k \mu_C(r)}\\
= -\frac{3}{16\pi \mu_C(r)}\frac{\log s_k}{s_k} -\frac{\log
4\sqrt{2\pi}+\frac32\log\rlog{r}}{8\pi s_k \mu_C(r)},\;r\in(a^\beta,a).
\end{split}
\end{equation}
For all $k$ large enough, by the implicit function theorem using (g0) and (g1), 
we have a unique solution $\tilde v_k$ that converges to zero uniformly on $C$, and so does its derivative.

2. Consider solutions $v_k^M$ of the equation \eqref{v} restricted to $r\in(0,1)\setminus C$, with the parameter $M\in\R$ to be determined at a later step. Comparing \eqref{v} and \eqref{inside}, we have
$\tilde v_k(r)\ge v_k^M(r)+\frac{\delta\log s_k}{s_k}$ at $r=a^\beta,a$,  with
some $\delta>0$. Let $\chi$ be a smooth function on $[0,\infty)$, which equals
$1$ on $[0,\frac12)$ and is supported in $[0,1)$, and let $\chi_0\in
C_0^\infty (0,\infty)$ be equal $1$ in a neighborhood of $[\rlog{a},
\beta\rlog{a}]$.
We define $v_k(e^{-t})$ as
\begin{equation}
\begin{cases}
\chi_0(t)[ \chi(s_k^{3/2}(\rlog{a}-t))\tilde v_k(e^{-t})+(1-\chi(s_k^{3/2}(\rlog{a})-t))v_k^{M_1}(e^{-t})], &t\in[0,\rlog{a}),
\\
 \tilde v_k(e^{-t}),&   t\in [\rlog{a}, \beta\rlog{a}],
 \\
 \chi_0(t)[\chi(s_k^{3/2}(t-\beta\rlog{a}))\tilde v_k(e^{-t})+(1-\chi(s_k^{3/2}(t-\beta\rlog{a})))v_k^{M_2}(e^{-t})], & t\in[\beta\rlog{a},\infty).
 \end{cases}
\end{equation}
We leave it to the reader to verify that $v_k$ satisfies \eqref{weaktow}
with the right hand side as in \eqref{deltawC}, once the values of $M_1$, $M_2$ are set to match the constants $p$, $q$, and, furtheremore, \eqref{ap3b}.
In order to show that $v_k\to 0$ in $H_0^1(B)$ we note that in the intervals $[\rlog{a}-s_k^{-1/2},\rlog{a}]$ and $[\beta\rlog{a},\beta\rlog{a}+s_k^{-1/2}]$,
\[
 |\frac{d}{dt}v_k(e^{-t})|^2\le s_k^{3}|\tilde v_k-v_k^M|^2+o(1)\le C(\log s_k)^2+o(1),
\]
which suffices to have a vanishing contribution from these intervals to the integral in $\|\nabla v_k\|_2^2$, while
on their complement $v_k'$ converges uniformly to zero. We set $\psi_k=\delta_{1/s_k}v_k$ which vanishes in $H_0^1(B)$, since so does $v_k$ and since each $\delta_{1/s_k}$ is an isometry.

3.  From (g2) it follows immediately that $\int F(u_k)\to 0$, and thus
\begin{equation}
 \begin{split}
J(u_k)=\frac12\|\nabla u_k\|_2^2+o(1)=\frac12\|\nabla \delta_{s_k}\mu_C\|_2^2+o(1)\\
\to  \frac12\|\nabla\mu_C\|_2^2=c.
 \end{split}
\end{equation}
An elementary explicit calculation yields $\|\nabla\mu_C\|_2^2=1+\frac{\log\beta}{4}$.
\end{proof}

\begin{remark}
Both results on existence of critical sequence with a given concentration
profile, Theorem~\ref{thm:critseq} and Theorem~\ref{badlevelsC}, can be easily
extended to the case, respectively, of a nodal profile with $|C|=0$, and of
a nodal profile when $C$ is a finite union of closed intervals, with only
elementary modifications of the proof. Furthermore, one can extend, in the nodal
setting, the gluing construction of Theorem~\ref{badlevelsC}
by the approximation reasoning of Theorem~\ref{thm:critseq} to obtain existence of a critical sequence of the form
$u_k=\delta_{1/s_k}\mu_{C_+,C_-}+\psi_k$ with $\|\nabla\psi_k\|_2\to 0$ at the
level $J=\frac12\|\nabla\mu_{C_+,C_-}\|_2^2$, for every Moser-Carleson-Chang
tower. We prefer, however, to defer such proof until we know of a new
application, interesting enough to provide due motivation.
\end{remark}

\section{Structure of Palais-Smale sequences}
We conclude the paper with a restriction of Theorem~ \ref{2cc-j} to critical sequences. Note that the case $c=\frac{m}{2}$ below reduces  verification of the Adimurthi-Struwe conjecture to the question if Moser functions are the only possible concentration profiles - which they are not.

\begin{theorem}
\label{thm:asymp}
Let $\Omega\subset\R^2$ be a bounded domain. Let $J$ be the functional \eqref{Jv} with $f$ of critical growth satisfying (g0), (g1) and (g2).
Let $u_k\in H_0^1(\Omega)$ be a bounded sequence such that
$J'(u_k)\to 0$ and $J(u_k)\to c$. Then the sequence $u_k$ has a renumbered subsequence
of the following form:

There exists an $m\in\N$, $m\le 2c$, sequences
$s_k^{(1)},\dots,s_k^{(m)}$ of positive numbers, convergent to zero for every $j=1,\dots,m$,
sequences $z_k^{(1)},\dots,z_k^{(m)}\in\bar\Omega$, $z_k^{(j)}\to z_j$, $j=1,\dots,m$, with $s_k^{(1)}=1$, $z_k^{(1)}=0$,
 and closed sets
$C_\pm^{(1)},\dots,C_\pm^{(m)}\in (0,1)$,
such that
\begin{equation}
\label{separates*}
z_p\neq z_q \text{ or } |\rlog{s_k^{(p)}}-\rlog{s_k^{(q)}}| \text{ whenever } p\neq q,
\end{equation}
\begin{equation}
\label{eq:ass}
u_k-\sum_{j=1}^m \delta_{s_k^{(j)}}\mu_{C_+^{(j)},C_-^{(j)}}(|\cdot-z_k^{(j)}|)\to 0 \text{ in } \exp L^2,
\end{equation}
and
\[
\|\nabla u_k\|_2^2\to\sum_j\|\nabla \mu_{C_+^{(j)},C_-^{(j)}}\|_2^2.
\]
If  $Z_j$ is the number of zeroes of $w^{(j)}=\mu_{C_+^{(j)},C_-^{(j)}}$, then $\sum_{j=1}^mZ_j<2c-2m$. In particular, if $c\le m$, all functions $w^{(j)}$ are sign definite.
Furthermore, if $\frac{m}{2}=c$, then for every $j=1,\dots,m$, $C^{(j)}=\{t_j\}$ for some $t_j\in (0,1)$, and $\mu_{C^{(j)}}$ is a Moser function $\mu_{t_j}$ and the convergence in \eqref{eq:ass} is in $H^1$-norm.

\end{theorem}

\begin{proof}
 The statement follows immediately from application of Theorem~\ref{lem:toy_pyramid} to Theorem~\ref{2cc-j} and properties of the profiles $\mu_{C_+,C_-}$ from Proposition~\ref{lem:LM}.
 Note that from (g2) it follows that $\int F(u_k)\to 0$, so $c\ge\frac12\sum_j\|\nabla \mu_{C_+^{(j)},C_-^{(j)}}\|_2^2$.
 Then relation $\sum_{j=1}^mZ_j<2c-2m$ is immediate from Proposition~\ref{lem:LM}.
 If $c=m/2$ then, necessarily, each of the norms in the right hand side equals
$1$, each $\mu_{C_+^{(j)},C_-^{(j)}}$ is a Moser function, the inequality
becomes an equality, and the resulting convergence of $H^1$-norms in
\eqref{eq:ass}, $\|\nabla u_k\|_2^2\to\sum \|\nabla
\mu_{C_+^{(j)},C_-^{(j)}}\|_2^2$, together with convergence in $\exp L^2$
implies $H^1$-convergence.  \end{proof}


\end{document}